\documentclass[a4paper,12pt]{amsart}  
\usepackage{a4wide,amsmath,amssymb,amsthm}     
\usepackage[utf8]{inputenc}
\usepackage{enumitem}
\usepackage{color}

\theoremstyle{plain}      
 
\newtheorem{thm}{Theorem}[section]     
     
\newtheorem{cor}[thm]{Corollary}

\newtheorem{lemma}[thm]{Lemma}     
\newtheorem{prop}[thm]{Proposition}

\theoremstyle{definition}      
\newtheorem{defn}[thm]{Definition}     
  
\newtheorem{example}[thm]{Example} 
\newtheorem{remark}[thm]{Remark}

\DeclareMathAlphabet{\doba}{U}{msb}{m}{n}

\def\Id{\mathrm{id}}

\def\T{\mathrm{T}}

\def\di{{\rm d}}
\def\tr{\mathrm{tr}}

\def\vol{{\mathrm{vol}}}
\def\ie{\emph{i.e.,} }

\let\p\partial
\let\<\langle 
\let\>\rangle

\newcommand{\definedas}{\mathrel{\raise.095ex\hbox{\rm :}\mkern-5.2mu=}}

\title{}

\author{Andrei Moroianu, Mihaela Pilca}

\address{Andrei Moroianu \\ Université Paris-Saclay, CNRS,  Laboratoire de mathématiques d'Orsay, 91405, Orsay, France, 
and Institute of Mathematics “Simion Stoilow” of the Romanian Academy, 21 Calea Grivitei, 010702 Bucharest, Romania}
\email{andrei.moroianu@math.cnrs.fr}

\address{Mihaela Pilca\\Fakult\"at f\"ur Mathematik\\
Universit\"at Regensburg\\Universit\"atsstr. 31 
D-93040 Regensburg, Germany}
\email{mihaela.pilca@mathematik.uni-regensburg.de}



\begin{document}   
	
	\title{Conformal product structures on compact Kähler manifolds}

	\begin{abstract} A conformal product structure on a Riemannian manifold is a Weyl connection with reducible holonomy. We give the geometric description of all compact Kähler manifolds admitting conformal product structures. 
		
	\end{abstract}
	\maketitle

	\section{Introduction}
	
	Every Riemannian manifold $(M,g)$ carries a unique torsion-free linear connection $\nabla$ compatible with the Riemannian metric, called the Levi-Civita connection. More generally, one can consider torsion-free linear connections on $M$, compatible with the conformal structure defined by the metric $g$. These are the so-called Weyl connections \cite{w1988} and they are in one-to-one correspondence with 1-forms on $M$ (see \cite{g1995}). The corresponding 1-form is called the Lee form of the Weyl connection with respect to $g$.
	
	Weyl connections can be {\em exact} (Levi-Civita connections of metrics conformal to $g$), {\em closed} (locally exact) or non-closed. In the first two cases  their restricted holonomy group is compact, whereas in the latter case, it is always non-compact. A classification of the possible restricted holonomy groups of Weyl connections can be obtained by the Berger-Simons Holonomy Theorem in the exact and closed cases, and by results of Merkulov and Schwachhöfer~\cite{ms1999} in the non-closed, irreducible case. 
	
	 Weyl connections with reducible holonomy define so-called {\em conformal product structures} \cite{bm2011}. The tangent bundle of a Riemannian manifold with conformal product structure can be written as the orthogonal direct sum of two integrable distributions, defining two conformal submersions in the neighborhood of each point. The {\em rank} of a conformal product structure is defined to be the smallest of the two dimensions of these distributions.
	 
	 When the Weyl connection defining a conformal product structure is closed but not exact, the structure is called locally conformally product (in short, LCP) \cite{f2024a}. The theory of LCP manifolds started with a first construction by Matveev and Nikolayevsky \cite{mn2015} giving a counterexample to a conjecture by Belgun and the first author \cite{bm2016}. A structure theorem was then obtained by Matveev and Nikolayevsky \cite{mn2017} in the analytic case and by Kourganoff \cite{k2019} in the general, smooth case. The study of compact LCP manifolds is currently a very active research topic (cf. \cite{adbm2024}, \cite{bfm2023}, \cite{dbm2024}, \cite{f2024b}, \cite{mp2024a}).
	 
	 A manifold with conformal product structure can be written locally as a product $M_1\times M_2$ and its Riemannian metric is of the form $e^{2f_1}g_1+e^{2f_2}g_2$, with $g_1$, $g_2$ Riemannian metrics on $M_1$, $M_2$ respectively, and $f_1,f_2\in C^\infty(M_1\times M_2)$ (cf. \cite{bm2011}, see also Proposition \ref{2.3} below). We thus see that {\em warped product metrics} (when $f_1\in C^\infty(M_2)$ and $f_2=0$) or {\em doubly warped product metrics} (when  $f_1\in C^\infty(M_2)$ and  $f_2\in C^\infty(M_1)$) are in fact very special cases of conformal product structures.
	 
	 In this paper we study conformal product structures on compact Kähler manifolds. This problem can be understood as a part of a wider program which intends to classify all compact Riemannian manifolds $(M,g)$ with special holonomy carrying a Weyl connection $D$ (different from the Levi-Civita connection) having special holonomy as well. Some parts of this program have been already carried out in \cite{m2019} when $D$ is exact, in \cite{bfm2023} and \cite{mmp2020} when $D$ is closed but non-exact, and in \cite{mp2024b} when $D$ is non-closed and reducible and $(M,g)$ is Einstein. 
	 
	The situation in real dimension 2 is very simple. We prove in Proposition \ref{dim2} that a compact 2-dimensional Kähler manifold (\ie an oriented Riemannian surface) carries a conformal product structure if and only if it is topologically a two-torus (\ie there is no restriction on the metric). 
	
	In our first main result (Theorem \ref{thm1}) we show that every conformal product structure on a compact Kähler manifold, different from the Levi-Civita connection, has rank 1. The argument goes roughly as follows: if $D$ is a Weyl connection on $(M,g)$ with reducible holonomy and $\T M=T_1\oplus T_2$ is a $D$-parallel orthogonal direct sum, then the orthogonal symmetry $S$ with respect to $T_2$ is $D$-parallel, so the curvature tensor of $D$ vanishes on $S$. This gives an expression for the Riemannian curvature tensor applied to $S$ in terms of the Lee form of $D$ and its covariant derivative. Using the symmetries of the Kähler curvature, taking suitable traces, and using the Cauchy-Schwarz inequality and Stokes' Theorem, we conclude that if the dimensions of $T_1$ and $T_2$ are at least 2, then the Lee form of $D$ with respect to $g$ vanishes.
	
	By similar arguments, we next show in Theorem \ref{thm2} that if $T_2$ has dimension 1 and the dimension of $M$ is $n\ge 4$, then the Lee form vanishes on $T_2$ (note that this result does not hold for $n=2$). 
	
	Finally, in Theorem \ref{thm3} we give two equivalent geometric characterizations of rank $1$ conformal product structures on compact K\"ahler manifolds: one  by providing an explicit expression of the pullback metric on the universal cover, and the second by the existence of a closed $1$-form of unit length vanishing on a $\nabla$-parallel complex subbundle of  $\T M$ of codimension $1$. Roughly speaking, we show that a compact Kähler manifold of real dimension at least 4 admits a conformal product structure other than the Levi-Civita connection if and only if it is locally isometric to a Riemannian product of a Kähler manifold of dimension $n-2$ and a two-torus carrying a geodesic foliation. 
	
	{\bf Acknowledgments.} This work was partly supported by the PNRR-III-C9-2023-I8 grant CF 149/31.07.2023 {\em Conformal Aspects of Geometry and Dynamics} and by the Procope Project No. 57650868 (Germany) / 48959TL (France). We warmly thank Vladimir Matveev for the example of a metric on the two-torus without geodesic foliation used in Proposition \ref{vm} below and to Thibault Lefeuvre for useful discussions. 
	
	\section{Preliminaries}
	\subsection{Weyl connections}
	Let $(M,g)$ be a Riemannian manifold with Levi-Civita connection denoted $\nabla$. We denote by $\sharp:\T^*M\to \T M$ and $\flat:\T M\to \T^* M$ the musical isomorphisms defined by the metric, which are $\nabla$-parallel and inverse to each other. In order to simplify notation, we will denote the metric $g$  by $\langle\cdot,\cdot\rangle$ and the associated norm by $\|\cdot\|$, when is no risk of confusion.

	\begin{defn}\label{weyl} A {\em Weyl connection} on $(M,g)$ is a torsion-free linear connection $D$  satisfying $Dg=-2\theta\otimes g$ for some $1$-form $\theta\in\Omega^1(M)$, called the {\em Lee form} of $D$ with respect to $g$. 
	\end{defn}
	
	The conformal analogue of the Koszul formula reads \cite{g1995}:
	\begin{equation}\label{ko} D_XY=\nabla_XY+\theta(Y)X+\theta(X)Y-\langle X, Y\rangle\theta^\sharp,\qquad\forall X,Y\in \Gamma(\T M).
	\end{equation}
	
	We thus see that there is a one-to-one correspondence between Weyl connections on $(M,g)$ and 1-forms on $M$: to each Weyl connection one associates its Lee form with respect to $g$ and to each 1-form one associates a Weyl connection by the formula \eqref{ko}.
	
	A Weyl connection is called closed (resp. exact) if its Lee form is closed (resp. exact). If $\theta=\di\varphi$ is exact, 
	$$D(e^{2\varphi}g)=e^{2\varphi}(2\di\varphi\otimes g+Dg)=e^{2\varphi}(2\theta\otimes g-2\theta\otimes g)=0,$$ 
	so $D$ is the Levi-Civita connection of $e^{2\varphi}g$. Thus exact Weyl connections are Levi-Civita connections of metrics in the conformal class of $g$, and closed Weyl connections are locally (but in general not globally) Levi-Civita connections of metrics conformal to $g$. 
	
	\subsection{Conformal product structures} Let $(M,g)$ be a Riemannian manifold of dimension $n\ge 2$.
	
	\begin{defn}\label{cp} A {\em conformal product structure} of rank $r\ge 1$ on $(M,g)$ is a Weyl connection $D$ together with an orthogonal $D$-parallel direct sum decomposition of the tangent bundle of $M$ as $\T M=T_1\oplus T_2$, with $\dim(T_1)\ge\dim(T_2)=r$. A conformal product structure $D$ on $(M,g)$ is {\em orientable} if the $D$-parallel distributions $T_1$ and $T_2$ are orientable.
	\end{defn}

Clearly, every conformal product structure induces by pull-back an orientable conformal product structure on some finite cover of $(M,g)$, with covering group contained in $(\mathbb{Z}/2\mathbb{Z})^2$. 
	
	Note that in \cite{bm2011} a different definition was given for conformal product structures, in terms of orthogonal conformal submersions (see also \cite{c2000}). However,  the two definitions are equivalent, by \cite[Theorem 4.3]{bm2011}. Moreover, a more concrete local characterization of the underlying Riemannian metrics on conformal products was indicated in \cite[\S 6.1]{bm2011}. For the sake of completeness, we restate it here:
	
	\begin{prop} \label{2.3} Let $(M_1,g_1)$ and $(M_2,g_2)$ be Riemannian manifolds, and denote by $M:=M_1\times M_2$ their product. For every $f_1,f_2\in C^\infty(M)$, consider the Riemannian metric $g:=e^{2f_1}g_1+e^{2f_2}g_2$ on $M$. Then there exists a unique Weyl connection which together with the decomposition $\T M=\T M_1\oplus \T M_2$ defines a conformal product structure on $(M,g)$. The Lee form of this Weyl connection with respect to $g$ is $\theta:=-\di^{M_1}f_1-\di^{M_2}f_2$.
	
	Conversely, a conformal product structure on $(M,g)$ is obtained by this construction in the neighborhood of each point. 
	\end{prop}
	
	\begin{proof}
	Assume first that $(M,g)=(M_1\times M_2,e^{2f_1}g_1+e^{2f_2}g_2)$. 
	For $i\in\{1,2\}$ we will consider arbitrary vector fields $X_i,\ Y_i,\ Z_i$  on $M_i$, identified with their canonical extension to $M$. The Koszul formula immediately gives $g(\nabla_{X_1}Y_1,Z_2)=-\frac12Z_2(g(X_1,Y_1))=-Z_2(f_1)g(X_1,Y_1)$. Let $\theta$ be any 1-form on $M$ and let $D$ be the Weyl connection with Lee form $\theta$ with respect to $g$. By the previous formula together with \eqref{ko} we obtain
	\begin{equation*}\begin{split}g(D_{X_1}Y_1,Z_2)&=g(\nabla_{X_1}Y_1,Z_2)+g(\theta(Y_1)X_1,Z_2)+g(\theta(X_1)Y_1,Z_2)-g(X_1,Y_1)\theta(Z_2)\\
	&=-(Z_2(f_1)+\theta(Z_2))g(X_1,Y_1).
	\end{split}\end{equation*}
	Similarly, we also get $g(D_{X_2}Y_2,Z_1)=-(Z_1(f_2)+\theta(Z_1))g(X_2,Y_2).$ This shows that $\T M_1$ and $\T M_2$ are $D$-parallel if and only if $\theta(Z_2)=-Z_2(f_1)$ and $\theta(Z_1)=-Z_1(f_2)$ for every $Z_1\in\T M_1$ and $Z_2\in \T M_2$, which is equivalent to $\theta:=-\di^{M_1}f_1-\di^{M_2}f_2$.
	
	Conversely, assume that $D$ is a Weyl connection on $(M,g)$ with Lee form $\theta$, preserving an orthogonal decomposition $\T M=T_1\oplus T_2$ of the tangent bundle. Since $D$ is torsion-free, $T_1$ and $T_2$ are integrable, so by the Frobenius Theorem, every point $x$ of $M$ has a connected neighborhood $U$ diffeomorphic to a product $U_1\times U_2$, such that $\T U_1=T_1$ and $\T U_2=T_2$ at each point of $U$. We write $x=(x_1,x_2)$, identify $U_1$ with $U_1\times \{x_2\}$ and $U_2$ with $\{x_1\}\times U_2$ and denote by $g_1$ the Riemannian metric on $U_1$ obtained by restricting the metric $g$ to the submanifold $U_1\times \{x_2\}$ of $U$ and by $g_2$ the Riemannian metric on $U_2$ obtained by restricting the metric $g$ to $\{x_1\}\times U_2$.
	
	Consider like before arbitrary vector fields $X_i,\ Y_i,\ Z_i$  on $U_i$, identified with their canonical extension to $U=U_1\times U_2$. Using again the Koszul formula and \eqref{ko} we get 
	\begin{equation*}\begin{split}0&=g(D_{X_1}Y_1,Z_2)=g(\nabla_{X_1}Y_1,Z_2)+g(\theta(Y_1)X_1,Z_2)+g(\theta(X_1)Y_1,Z_2)-g(X_1,Y_1)\theta(Z_2)\\
	&=-\frac12 Z_2(g(X_1,Y_1))-\theta(Z_2)g(X_1,Y_1),
	\end{split}\end{equation*}
	showing that 
	\begin{equation}\label{gd} Z_2(g(X_1,Y_1))=-2\theta(Z_2)g(X_1,Y_1),\qquad \forall X_1,Y_1\in \Gamma(\T U_1),\ \forall Z_2\in \Gamma(\T U_2).
	\end{equation}
	For every point $y_2\in U_2$, let $\gamma:[0,1]\to U_2$ be a smooth curve with $\gamma(0)=x_2$ and $\gamma(1)=y_2$. We pick any $y_1\in U_1$ and denote by $f(t):=g(X_1,Y_1)_{(y_1,\gamma(t))}.$ By definition we have $f(0)=g_1(X_1,Y_1)_{y_1}$ and by \eqref{gd} applied to $Z_2=\dot\gamma$ we get $$f'(t)=-2\theta(\dot\gamma(t))f(t).$$ 
	
	Solving this differential equation gives 
	$$g(X_1,Y_1)_{(y_1,y_2)}=f(1)=e^{\int_0^1-2\theta(\dot\gamma(t))\di t}f(0)=e^{\int_0^1-2\theta(\dot\gamma(t))\di t}g_1(X_1,Y_1)_{y_1}.$$

	Since this formula holds for arbitrary vector fields $X_1,Y_1$ on $U_1$ and the function $e^{\int_0^1-2\theta(\dot\gamma(t))\di t}$ is independent of $X_1,Y_1$, we deduce that $g|_{T_1}$ is proportional to $g_1$ at every point $(y_1,y_2)\in U$, \ie there exists a function $f_1\in C^\infty(U)$ such that $g|_{T_1}=e^{2f_1}g_1$. Similarly, there exists a function $f_2\in C^\infty(U)$ such that $g|_{T_2}=e^{2f_2}g_2$, and since moreover $T_1$ and $T_2$ are orthogonal with respect to $g$, we finally obtain that $g=e^{2f_1}g_1+e^{2f_2}g_2$.
		\end{proof}
	
	\subsection{Metric characterization of conformal product structures}
Let $(M, g)$ be a Riemannian manifold. Recall that for any two vectors $X, Y\in\T M$, one can define a symmetric endomorphism $X\odot Y$ as follows: 
$$(X\odot Y)(Z):=\langle X,Z\rangle Y+\langle Y, Z\rangle X,\qquad\forall Z\in\T M.$$
We show the following metric characterization of  conformal product structures: 
\begin{lemma}\label{lemma_cps}	
	Let $(M, g)$  be a Riemannian manifold with Levi-Civita connection $\nabla$. The following assertions are equivalent:
	\vspace{-0.2cm}
	\begin{enumerate}[label=(\roman*)]
		\item There exists a conformal product structure on $(M,g)$.
		\item There exists a symmetric involution $S$ of $\T M$ different from $\pm\Id_{\T M}$ and a $1$-form $\theta$ on $M$, such that 
		\begin{equation}\label{derivS}
			\nabla_X S=SX\odot \theta^\sharp - S\theta^\sharp\odot X, \quad \forall X\in\Gamma(\T M).
		\end{equation}
		\end{enumerate}	
\end{lemma}	
\begin{proof}
$(i) \Longrightarrow (ii)$ Assume that $(M,g)$ is endowed with a conformal product structure $D$ whose orthogonal $D$-parallel decomposition is $\T M=T_1\oplus T_2$.
We define $S$ to be the symmetric involution of $\T M$ given by:
\begin{equation}\label{S}
S|_{T_1}=\Id_{T_1}, \quad S|_{T_2}=-\Id_{T_2}.
\end{equation}
Denoting by $\theta$ the Lee form of $D$ with respect to $g$, we compute for all $X,Y\in \Gamma(\T M)$:
 \begin{equation}\label{weylS1}
 	\begin{split}
 		(D_X S)(Y)=&D_X(SY)-S(D_X Y)\\
 		\stackrel{\eqref{ko}}{=}&\nabla_X(SY)+\theta(SY)X+\theta(X)SY-\langle X, SY\rangle\theta^\sharp\\
 		&-S(\nabla_X Y+\theta(Y)X+\theta(X)Y-\langle X, Y\rangle\theta^\sharp)\\
 		=&(\nabla_X S)(Y)+\theta(SY)X-\langle X, SY\rangle\theta^\sharp-\theta(Y)SX+\langle X, Y\rangle S\theta^\sharp.
 	\end{split}	
 \end{equation}	
Since the endomorphism $S$ is $D$-parallel and $S$ is symmetric, \eqref{weylS1} gives:
\begin{equation}\label{weylS2}
	\begin{split}
		(\nabla_X S)(Y)&=\langle SX, Y\rangle\theta^\sharp+\theta(Y)SX-g(S\theta^\sharp, Y)X-\langle X, Y\rangle S\theta^\sharp,\qquad\forall X,Y\in \Gamma(\T M).
	\end{split}	
\end{equation}	
Using the notation introduced above for the symmetric product, \eqref{weylS2} is equivalent to \eqref{derivS}.

$(ii) \Longrightarrow (i)$ Conversely, let $S$ be a symmetric involution of $\T M$ different from $\pm\Id_{\T M}$ and let $\theta$ be a $1$-form, such that the covariant derivative of $S$ is given by the formula \eqref{derivS}. If $T_1$ denotes the $(+1)$-eigenbundle of $S$ and $T_2$ the $(-1)$-eigenbundle of~$S$, then we have the direct sum decomposition $\T M=T_1\oplus T_2$. Note that this decomposition is non-trivial because $S\neq\pm\Id_{\T M}$  and it is orthogonal, because $S$ is a symmetric involution, so in particular an orthogonal endomorphism. Let $D$ be the Weyl connection whose  Lee $1$-form with respect to $g$ is $\theta$. Performing the above computations backwards shows that \eqref{derivS} implies $D_X S=0$, for all $X\in \Gamma(\T M)$, so that the eigenbundles $T_1$ and $T_2$ are $D$-parallel. Hence, the Weyl connection $D$ together with the orthogonal $D$-parallel decomposition $\T M=T_1\oplus T_2$ define a conformal product structure on~$(M,g)$. 
\end{proof}

\section{Conformal product structures on Riemannian surfaces}
A Kähler manifold of real dimension 2 is just an oriented Riemannian surface.
In this section we characterize compact oriented Riemannian surfaces admitting conformal product structures (which by definition are necessarily of rank 1). The result is very simple: 

\begin{prop}\label{dim2} A compact oriented Riemannian surface $(M,g)$ carries a conformal product structure if and only if its genus is equal to $1$. 
\end{prop}
\begin{proof} If $(M,g)$ has a conformal product structure, its tangent bundle carries a 1-dimensional distribution. This shows that $M$ has vanishing Euler characteristic, so its genus is 1, \ie $M$ is a two-torus.

Conversely, if $M$ has genus 1, then it carries a nowhere vanishing vector field, so there exists a vector field $\xi$ which has unit length with respect to $g$. Let $J$ denote the Kähler structure of $(M,g)$ (which is well defined since $M$ is oriented). Then $\nabla_X\xi$ is orthogonal to $\xi$ for every $X\in \T M$, so there exists $\alpha\in\Omega^1(M)$ such that $\nabla_X\xi=\alpha(X)J\xi$ for every $X\in \T M$.

We claim that the Weyl connection on $(M,g)$ with Lee form $\theta:=J\alpha$ preserves the two orthogonal distributions $T_1:=\mathbb{R}\xi$ and $T_2:=\mathbb{R}J\xi$. Indeed, for every $X\in \T M$ we can write using \eqref{ko}
\begin{equation*}\begin{split}
	\langle D_X\xi,J\xi\rangle&=\langle\nabla_X\xi,J\xi\rangle+\theta(X)	\langle\xi,J\xi\rangle+\theta(\xi)\langle X,J\xi\rangle-\langle X,\xi\rangle\theta(J\xi)\\
&=\alpha(X)-\alpha(J\xi)\langle X,J\xi\rangle-\alpha(\xi)\langle X,\xi\rangle=0,
\end{split}\end{equation*}
since $\{\xi,J\xi\}$ is an orthonormal basis of $(\T M,g)$ at any point. Thus $D$ preserves the distribution $T_1$ and since every Weyl connection preserves orthogonality, $D$ also preserves $T_2$, showing that $(M,g)$ has a conformal product structure of rank 1.
\end{proof}
	
	\section{Conformal product structures on higher dimensional Kähler manifolds}
	
	In this section we describe compact Kähler manifolds $(M,g,J)$ of real dimension $n\ge 4$ carrying conformal product structures. By definition, a Riemannian manifold $(M,g)$ carries a conformal product structure with vanishing Lee form if and only if $(M,g)$ has reducible holonomy representation. We will thus restrict ourselves to Weyl connections different from the Levi-Civita connection of $g$, \ie with non-identically zero Lee form. Our first main result can be stated as follows:
	
	\begin{thm}\label{thm1} Let $(M,g,J)$ be a compact Kähler manifold. Then every conformal product structure on $(M,g)$ with non-identically zero Lee form has rank~$1$. 
\end{thm}
	
	\begin{proof}
Let $(M,J,g)$  be a compact K\"ahler manifold of real dimension $n$ endowed with a conformal product structure of rank $r$ with non-identically zero Lee form. Note that for $n=2$ we have automatically $r=1$, so the result is non-empty only for $n\ge 4$.

As before, we denote by $D$ the associated Weyl connection, by $\theta$ its Lee form with respect to $g$ and consider the orthogonal $D$-parallel decomposition $\T M=T_1\oplus T_2$ with $\dim(T_1)\geq \dim(T_2)=r$. Let $S$ be the $D$-parallel symmetric involution defined by Lemma \ref{lemma_cps}. We first compute $R_{X,Y} S$, where $R$ is the Riemannian curvature tensor of $g$.  In order to simplify notation, we identify in the computations below vectors and 1-forms using the metric $g$ and denote by \mbox{$T:=\nabla\theta$} the endomorphism defined by the covariant derivative of the Lee form $\theta$. For any vector fields $X$ and $Y$ which are parallel at the point where the computation is done, we obtain:
\begin{equation*}
	\begin{split}
		\nabla^2_{X,Y} S\overset{\eqref{derivS}}{=}&\nabla_X(SY\odot \theta - S\theta\odot Y)\\
		=&(\nabla_X S)Y\odot \theta +SY \odot  TX
		-(\nabla_X S)\theta\odot Y-STX\odot Y\\
		\overset{\eqref{derivS}}{=}&(SX\odot \theta - S\theta\odot X)(Y)\odot \theta+SY \odot  TX\\
		&-(SX\odot \theta - S\theta\odot X)(\theta)\odot Y-STX\odot Y\\
		=&\<SX,Y\>\theta\odot\theta+\<\theta, Y\>SX\odot \theta-\<S\theta, Y\>X\odot\theta-\<X,Y\>S\theta\odot \theta +SY \odot  TX \\
		&-\<SX,\theta\>\theta\odot Y-\|\theta\|^2SX\odot Y+\<S\theta, \theta\>X\odot Y+\<X,\theta\>S\theta \odot Y -STX\odot Y.
	\end{split}	
\end{equation*}	
Since $R_{X,Y} S=\nabla^2_{X,Y} S-\nabla^2_{Y,X} S$, exchanging $X$ and $Y$ in the above formula, and using the symmetry of $S$ yields after some simplifications:
\begin{equation}\label{curvS}
	\begin{split}
		R_{X,Y} S	=&SY \odot  TX-SX \odot  TY+STY\odot X-STX\odot Y\\
		&+\<\theta, Y\>SX\odot \theta-\<\theta, Y\>S\theta \odot X+\<\theta, X\>S\theta \odot Y\\
		&-\<\theta, X\>SY\odot \theta -\|\theta\|^2(SX\odot Y-SY\odot X).
	\end{split}	
\end{equation}	

Applying this formula to $JX$ and $JY$, we also obtain:	
\begin{equation}\label{curvJ_S}
	\begin{split}
		R_{JX,JY} S=&SJY \odot  TJX-SJX \odot  TJY+STJY\odot JX-STJX\odot JY\\
		&+\<\theta, JY\>SJX\odot \theta-\<\theta, JY\>S\theta \odot JX+\<\theta, JX\>S\theta \odot JY\\
		&-\<\theta, JX\>SJY\odot \theta-\|\theta\|^2(SJX\odot JY-SJY\odot JX).
	\end{split}	
\end{equation}	

Since $R$ is the curvature of a K\"ahler metric, we have that $R_{X,Y}S=R_{JX, JY}S$ for every $X,Y\in \T M$. Let $\{e_i\}_{i=1,n}$  be a local orthonormal basis of $\T M$. For every $Y\in \T M$ we have $\displaystyle \sum_{i=1}^n(R_{e_i, Y}S)(e_i)=\sum_{i=1}^n(R_{Je_i, JY}S)(e_i)$, which
according to \eqref{curvS} and \eqref{curvJ_S} yields:
\begin{equation*}
	\begin{split}
		&TSY +\mathrm{tr}(T) SY-\mathrm{tr}(S) TY-STY+(n+1)STY-\mathrm{tr}(ST) Y-STY\\
		&+\<\theta, Y\>(\mathrm{tr}(S)\theta+S\theta)-(n+1)\<\theta, Y\>S\theta+\<S\theta, \theta\>Y+\<\theta,Y\>S\theta-\<SY,\theta\>\theta-\|\theta\|^2SY\\
		&-\|\theta\|^2(\mathrm{tr}(S)Y+SY-(n+1)SY)\\
		=&TJSJY+\mathrm{tr}(TJ)SJY-\mathrm{tr}(SJ)TJY-SJTJY+JSTJY-\mathrm{tr}(STJ)JY+STY\\
		&+\<\theta, JY\>(\mathrm{tr}(SJ)\theta+SJ\theta)-\<\theta, JY\>JS\theta-\<S\theta, J\theta\>JY-\<\theta,Y\>S\theta+\<J\theta, SJY\>\theta\\
		&-\|\theta\|^2(\mathrm{tr}(SJ)JY-SY-JSJY).
	\end{split}	
\end{equation*}		
Using the fact that $\mathrm{tr}(SJ)=0$, the above equality reads after simplification:
\begin{equation}\label{curv_Y_S}
	\begin{split}
		&TSY +\mathrm{tr}(T) SY-\mathrm{tr}(S) TY+(n-1)STY-\mathrm{tr}(ST) Y\\
		&+\<\theta, Y\>(\mathrm{tr}(S)\theta+(1-n)S\theta)+\<S\theta, \theta\>Y-\<SY,\theta\>\theta\\
		&+\|\theta\|^2((n-1)SY-\mathrm{tr}(S)Y)\\
		=&TJSJY+\mathrm{tr}(TJ)SJY-SJTJY+JSTJY-\mathrm{tr}(STJ)JY+STY\\
		&-\<J\theta, Y\>SJ\theta+\<J\theta, Y\>JS\theta-\<S\theta, J\theta\>JY-\<\theta,Y\>S\theta-\<JSJ\theta, Y\>\theta\\
		&+\|\theta\|^2(SY+JSJY)
	\end{split}	
\end{equation}		

Taking the scalar product with $SY$ in this formula, and summing over  a local orthonormal basis $Y=e_i$  yields:
\begin{equation*}
	\begin{split}
		&\mathrm{tr}(T)+n\mathrm{tr}(T)-\mathrm{tr}(S)\mathrm{tr}(ST)+(n-1)\mathrm{tr}(T)-\mathrm{tr}(ST)\mathrm{tr}(S)\\&+\<\theta, S\theta\>\tr(S)+(1-n)\|\theta\|^2+\<S\theta, \theta\>\tr(S)-\|\theta\|^2+\|\theta\|^2(n(n-1)-(\tr(S))^2)\\
		=&\tr(STJSJ)+\tr(T)+\tr(SJSTJ)+\tr(T)-\|\theta\|^2-\<SJSJ\theta, \theta\>-\|\theta\|^2-\<SJSJ\theta, \theta\>\\
		&+\|\theta\|^2(n+\tr(SJSJ)),
	\end{split}	
\end{equation*}			
which further simplifies to
\begin{equation}\label{trace}
	\begin{split}
		0=&2(n-1)\mathrm{tr}(T)-2\mathrm{tr}(S)\mathrm{tr}(ST)-2\tr(JSJST)+2\tr(S)\<\theta, S\theta\>\\&+2\<\theta, JSJS\theta\>+\|\theta\|^2\left(n^2-3n+2-(\tr(S))^2-\tr(JSJS)\right).
	\end{split}	
\end{equation}		

We will now interpret the traces occurring in \eqref{trace} in terms of codifferentials of $1$-forms on $(M,g)$.
\begin{lemma}\label{lemmatraces}
	The traces of the endomorphisms $T$, $ST$ and $JSJST$ of $\T M$ can be expressed as follows:
	\vspace{-0.3cm}
	\begin{enumerate}
		\item[(i)] $\tr(T)=-\delta\theta$.
		\item[(ii)]  $\tr(ST)=-\delta(S\theta)-\tr(S)\|\theta\|^2+n\<\theta, S\theta\>$.
		\item[(iii)]  $\tr(JSJST)=-\delta(JSJS\theta)+\<\theta, JSJS\theta\>+\|\theta\|^2-\|\theta\|^2\tr(JSJS)-\tr(S)\<\theta, S\theta\>$.
	\end{enumerate}
\end{lemma}	

\begin{proof}
	$(i)$ By the definition of $T:=\nabla\theta$, it follows that $\tr(T)=-\delta\theta$.
	
	$(ii)$ If $\{e_i\}_{i=1,n}$  is a local orthonormal basis of $\T M$, we compute:
		\begin{equation*}
			\begin{split}
				\tr(ST)&= \sum_{i=1}^n \<ST e_i, e_i\>=\sum_{i=1}^n \<S\nabla_{e_i}\theta, e_i\>=\sum_{i=1}^n \<\nabla_{e_i}(S\theta), e_i\>-\sum_{i=1}^n \<(\nabla_{e_i}S) \theta, e_i\>\\
				&\overset{\eqref{derivS}}{=}-\delta(S\theta)-\sum_{i=1}^n\<(Se_i\odot\theta -S\theta\odot e_i)(\theta), e_i\>\\
				&=-\delta(S\theta)-\sum_{i=1}^n\left(\<Se_i, \theta\>\<\theta,e_i\>+\|\theta\|^2\<Se_i, e_i\>-\<S\theta, \theta\>\<e_i, e_i\>-\<e_i,\theta\>\<S\theta, e_i\>\right)\\
				&=-\delta(S\theta)-\tr(S)\|\theta\|^2+n\<\theta, S\theta\>.
			\end{split}
		\end{equation*}
		
$(iii)$ In order to obtain the formula for the trace of $JSJST$ we start by computing separately:
		\begin{equation*}
			\begin{split}
				\<\nabla_{e_i}(SJS)\theta, Je_i\>=&\<(\nabla_{e_i}S)(JS \theta)+SJ(\nabla_{e_i}S)\theta, Je_i\>\\
				\overset{\eqref{derivS}}{=}&\<(Se_i\odot\theta -S\theta\odot e_i)(JS\theta)+SJ(Se_i\odot\theta -S\theta\odot e_i)(\theta), Je_i\>\\
				=&\<Se_i, JS\theta\>\<\theta, Je_i\>-\<\theta, JS\theta\>\<Se_i, Je_i\>-\<JS\theta, e_i\>\<S\theta, Je_i\>\\
				&+\<Se_i, \theta\>\<SJ\theta, Je_i\>+\|\theta\|^2\<SJSe_i, Je_i\>\\
				&-\<\theta, S\theta\>\<SJe_i, Je_i\>-\<e_i,\theta\>\<SJS\theta, Je_i\>.
			\end{split}
		\end{equation*}
		Summing up over $i$ and taking into account that $\tr(SJ)=0$ we obtain:
		$$\sum_{i=1}^n \< \nabla_{e_i}(SJS)\theta, Je_i\>=\<\theta, JSJS\theta\>+\|\theta\|^2-\|\theta\|^2\tr(JSJS)-\tr(S)\<\theta, S\theta\>.$$
		Using this, we then get:
		\begin{equation*}
			\begin{split}
				\tr(JSJST)&= \sum_{i=1}^n \<JSJST e_i, e_i\>=\sum_{i=1}^n \<JSJS\nabla_{e_i}\theta, e_i\>\\
				&=\sum_{i=1}^n \<\nabla_{e_i}(JSJS\theta), e_i\>-\sum_{i=1}^n \<(\nabla_{e_i}JSJS)\theta, e_i\>\\
				&=-\delta(JSJS\theta)+\sum_{i=1}^n \< \nabla_{e_i}(SJS)\theta, Je_i\>\\
				&= -\delta(JSJS\theta)+\<\theta, JSJS\theta\>+\|\theta\|^2-\|\theta\|^2\tr(JSJS)-\tr(S)\<\theta, S\theta\>.
			\end{split}
		\end{equation*}
\end{proof}

According to Lemma \ref{lemmatraces}, Equality \eqref{trace} reads
\begin{equation*}
	\begin{split}
		0=&-2(n-1)\delta(\theta)-2\mathrm{tr}(S)\left(-\delta(S\theta)-\tr(S)\|\theta\|^2+n\<\theta, S\theta\> \right)\\
		&-2\left(-\delta(JSJS\theta)+\<\theta, JSJS\theta\>+\|\theta\|^2-\|\theta\|^2\tr(JSJS)-\tr(S)\<\theta, S\theta\>\right)\\
		&+2\tr(S)\<\theta, S\theta\>+2\<\theta, JSJS\theta\>+\|\theta\|^2\left(n^2-3n+2-(\tr(S))^2-\tr(JSJS)\right),
	\end{split}	
\end{equation*}		
or, equivalently
\begin{equation}\label{trace2}
	\begin{split}
		0=&\|\theta\|^2\left(n^2-3n+(\tr(S))^2+\tr(JSJS)\right)-2(n-2)\tr(S)\<\theta, S\theta\>\\
		&-2(n-1)\delta(\theta)+2\mathrm{tr}(S)\delta(S\theta)+2\delta(JSJS\theta).
	\end{split}	
\end{equation}		

We now remark that  $JS$ and $SJ$ are orthogonal endomorphisms of $\T M$, and thus $\|JS\|=\|SJ\|=\sqrt{n}$. The Cauchy-Schwarz inequality then yields
\begin{equation}\label{ineq}
	\tr(JSJS)=\<JS, (JS)^t\>=-\<JS, SJ\>\geq -\|JS\|\|SJ\|= -n,
\end{equation}		
 Since $JS$ and $SJ$ have the same norm, equality holds in \eqref{ineq} if and only if $SJ=JS$, in which case we have $JSJS=JSSJ=-\mathrm{Id}_{\T M}$. 

We now integrate \eqref{trace2} over the compact manifold $M$ against the volume form $\vol^g$. Since $\tr(S)=n-2r$ is constant, the integral of each of the last three terms vanishes by Stokes' Theorem. Moreover, $n-r\geq r$ by assumption, so applying the Cauchy-Schwarz inequality $|\<\theta, S\theta\>|\leq \|\theta\|^2$ and the inequality \eqref{ineq} yields:
\begin{equation}\label{ineqtheta1}
	\begin{split}
		0=&\int_M \left(\|\theta\|^2\left(n^2-3n+(\tr(S))^2+\tr(JSJS)\right)-2(n-2)\tr(S) \<\theta, S\theta\>\right)\vol^g\\
		\geq& \int_M \left(n^2-3n+(n-2r)^2-n-2(n-2)(n-2r)\right)\|\theta\|^2\vol^g\\
		=&\int_M 4r(r-2)\|\theta\|^2\vol^g.
	\end{split}
\end{equation}

We now assume for a contradiction that $r\geq 2$ and distinguish the following two cases:

1) If $r\geq 3$, then $4r(r-2)>0$ and \eqref{ineqtheta1} implies that $\|\theta\|=0$, so $\theta$ vanishes everywhere on~$M$, contradicting the hypothesis.

2) If $r=2$, then equality holds in \eqref{ineqtheta1}, which implies in particular that $\tr(JSJS)=-n$. Thus we have equality in \eqref{ineq}, showing that $JSJS=-\mathrm{Id}_{\T M}$. Replacing this into the formula for the trace of $JSJST=-T$ in Lemma~\ref{lemmatraces}(iii), we obtain:
$$-\tr(T)=\delta(\theta)+n\|\theta\|^2-\tr(S)\<\theta, S\theta\>.$$
Since $\tr(T)=-\delta(\theta)$, it follows that $n\|\theta\|^2-\tr(S)\<\theta, S\theta\>=0$. The Cauchy-Schwarz inequality further yields
$$0=n\|\theta\|^2-\tr(S)\<\theta, S\theta\>\geq n\|\theta\|^2-(n-2r)\|\theta\|^2=2r\|\theta\|^2=4\|\theta\|^2,$$
which implies that $\theta$ vanishes everywhere on~$M$, contradicting the hypothesis in this case as well.

Consequently, the rank of the conformal product structure has to be $1$.
	 \end{proof}
	 We now obtain a further restriction on the Lee form of conformal product structures on compact Kähler manifolds when the real dimension is at least 4. 
	 
	 \begin{thm}\label{thm2} Let $(M,g,J)$ be a compact Kähler manifold of real dimension $n\ge 4$ and let $D$ be a conformal product structure of rank $1$ with $D$-parallel orthogonal decomposition $\T M=T_1\oplus T_2$, where $\dim(T_2)=1$. Then the Lee form of $D$ with respect to~$g$ vanishes on $T_2$. 
\end{thm}
	
	\begin{proof}
Let $S$ be the symmetric endomorphism defined by \eqref{S}. Then $\tr(S)=n-2$. By replacing $M$ with a double cover if necessary, one can assume that $T_2$ is orientable. Let $\xi$ be a vector spanning $T_2$. Then $J\xi$ is orthogonal on $\xi$ and since $\dim(T_2)=1$ and $T_1$ and $T_2$ are orthogonal, it follows that $J\xi\in T_1$, so $S(J\xi)=J\xi$. We then compute:
$$JSJS(\xi)=JSJ(-\xi)=-JS(J\xi)=-J(J\xi)=\xi.$$
$$JSJS(J\xi)=JSJ(J\xi)=-JS(\xi)=J\xi.$$
For each $X\in T_1$ which is orthogonal to $J\xi$, it follows that $JX\in  T_1$ and we obtain:
$$JSJS(X)=JSJ(X)=J(JX)=-X.$$
We conclude that $JSJS$ is an involution of $\T M$ with $(+1)$-eigenspace of dimension $2$ and $(-1)$-eigenspace of  dimension $n-2$, so $\tr(JSJS)=4-n$. If $\theta$ denotes as above the Lee form of $D$ with respect to $g$, Equality~\eqref{trace2} reads in this case:
\begin{equation}\label{trace3}
	\begin{split}
		0=&2(n-2)^2(\|\theta\|^2-\<\theta, S\theta\>)-2(n-1)\delta(\theta)+2(n-2)\delta(S\theta)+2\delta(JSJS\theta).
	\end{split}	
\end{equation}		
Let us decompose $\theta=\theta_1+\theta_2$, where $\theta_i:=\theta|_{T_i}$, for $i\in\{1,2\}$. Then $S(\theta_1)=\theta_1$ and 	$S(\theta_2)=-\theta_2$ and thus we have: 
$$\|\theta\|^2-\<\theta, S\theta\>=(\|\theta_1\|^2+\|\theta_2\|^2)-(\|\theta_1\|^2-\|\theta_2\|^2)=2\|\theta_2\|^2.$$
Replacing this into \eqref{trace3} and integrating  over $M$ yields $0=\displaystyle\int_M 4(n-2)^2\|\theta_2\|^2\vol^g$. Since $n\geq 4$, we conclude that $\theta_2=0$, which means that the Lee form $\theta$ vanishes on $T_2$.
	 \end{proof}
	 We are now ready for the geometric characterization of conformal product structures on compact Kähler manifolds of dimension $n\ge 4$. From the previous result, we can restrict to the rank 1 case.
	 
	 	\begin{thm}\label{thm3} Let $(M,g,J)$ be a compact Kähler manifold of real dimension $n\ge 4$ with Levi-Civita connection denoted by $\nabla$. Then the following assertions are equivalent:
	 	
(i) $(M,g)$ carries an orientable conformal product structure of rank $1$.

(ii) The universal cover $(\widetilde M,\widetilde g)$ is isometric to $(\mathbb{R}^2\times K,e^{-2\varphi}\di s^2+\di t^2+g_K)$, where $(K,g_K)$ is a simply connected Kähler manifold, $\varphi\in C^\infty(\mathbb{R}^2)$, and every element of $\pi_1(M)$ has the form 
\begin{equation}\label{pi1}\gamma(s,t,x)=(\gamma_0(s),t+t_0,\gamma_K(x)),\ \qquad\forall (s,t)\in\mathbb{R}^2,\ \forall x\in K,
\end{equation}
for some function $\gamma_0\in C^\infty(\mathbb{R})$, some constant $t_0\in \mathbb{R}$ and some isometry $\gamma_K$ of $(K,g_K)$.

(iii) The tangent bundle of $M$ has a $\nabla$-parallel complex subbundle of codimension $1$, contained in the kernel of a closed $1$-form of unit length.
\end{thm}
	
	\begin{proof}
$(i) \Longrightarrow (ii)$ Assume that $\T M=T_1\oplus T_2$ is an orthogonal $D$-parallel decomposition, with $\dim(T_2)=1$. By the orientability assumption, there exists a unit length vector field $\xi$ spanning $T_2$. Then $D_X \xi=\alpha(X)\xi$ for some 1-form $\alpha$ on $M$. Moreover,  $\theta|_{T_2}=0$ by Theorem \ref{thm2}, so $\theta(\xi)=0$. Using this, together with \eqref{ko}, we get 
	$$\alpha(X)\xi=\nabla_X\xi+\theta(\xi)X+\theta(X)\xi-\langle\xi,X\rangle\theta^\sharp=\nabla_X\xi+\theta(X)\xi-\langle X,\xi\rangle\theta^\sharp,$$
	whence $\nabla_X\xi=(\alpha-\theta)(X)\xi+g(X,\xi)\theta^\sharp$. Since $\xi$ has unit length, $g(\nabla_X\xi,\xi)=0$, so $\alpha-\theta=0$, thus showing that
	\begin{equation}\label{nxi} \nabla_X\xi=\langle X,\xi\rangle\theta^\sharp,\qquad\forall X\in \T M.
	\end{equation}
	
	Denoting as before by $T:=\nabla\theta^\sharp$ the endomorphism of $\T M$ defined by the covariant derivative of $\theta^\sharp$, we readily obtain from \eqref{nxi}:
	\begin{equation}\label{rxi} R_{X,Y}\xi=\theta(Y)\langle X,\xi\rangle\theta^\sharp-\theta(X)\langle Y,\xi\rangle \theta^\sharp+\langle Y,\xi \rangle TX-\langle X,\xi\rangle TY,\qquad\forall X,Y\in \T M.
	\end{equation}
	Using the fact that $R_{X,Y}\xi=R_{JX,JY}\xi$ since $(M,g,J)$ is Kähler, and applying \eqref{rxi} twice, for $(X,\xi)$ and $(JX,J\xi)$, we obtain:
	\begin{equation}\label{t1}-\theta(X)\theta^\sharp+TX-\langle X,\xi\rangle T\xi=\theta(J\xi)\langle JX,\xi\rangle\theta^\sharp-\langle JX,\xi \rangle TJ\xi,\qquad\forall X\in \T M.
		\end{equation}
	
	On the other hand, on any Kähler manifold we have $\di^c\delta+\delta \di^c=0$ and $\di\delta^c+\delta^c \di=0$. Introducing the notation  $b:=\theta(J\xi)$, we get from \eqref{nxi}:
\begin{equation}\label{dixi}\di\xi^\flat=\xi^\flat\wedge\theta,\qquad \di^c\xi^\flat=J\xi^\flat\wedge\theta,\qquad \delta\xi^\flat=-\theta(\xi)=0,\qquad \delta^c\xi^\flat=-\theta(J\xi)=-b,
\end{equation}
	whence using a local orthonormal basis $\{e_i\}_{i=1,n}$ of $\T M$ we infer:
	\begin{equation*}\begin{split}0&=\delta \di^c\xi^\flat=\delta(J\xi^\flat\wedge\theta)=-\sum_{i=1}^n e_i\lrcorner (\nabla_{e_i}(J\xi^\flat\wedge\theta))\\
	&=-\sum_{i=1}^n e_i\lrcorner (\langle e_i,\xi \rangle J\theta\wedge\theta+J\xi^\flat \wedge (Te_i)^\flat)=b\theta-(TJ\xi)^\flat+\tr(T)J\xi^\flat,
	\end{split}\end{equation*}
	showing that
	\begin{equation}\label{t2}
	TJ\xi=b\theta^\sharp+\tr(T)J\xi.
		\end{equation}
Similarly, we get from \eqref{dixi}
	\begin{equation*}\begin{split}0&=(\di\delta^c+\delta^c \di)\xi^\flat=-\di b+\delta^c(\xi^\flat\wedge\theta)=-\di b-\sum_{i=1}^n Je_i\lrcorner (\nabla_{e_i}(\xi^\flat\wedge\theta))\\
	&=-\di b-\sum_{i=1}^n Je_i\lrcorner (\xi^\flat \wedge (Te_i)^\flat)=-\di b+(TJ\xi)^\flat-\tr(JT)\xi^\flat,
	\end{split}\end{equation*}
	showing that
	\begin{equation}\label{t3}\di b=(TJ\xi)^\flat-\tr(JT)\xi^\flat.
	\end{equation}
	In particular, applying this to $\xi$ and using \eqref{t2} yields 
	\begin{equation}\label{xib}\xi(b)=-\tr(JT).
	\end{equation}
	From \eqref{t1} and \eqref{t2} we obtain
	\begin{equation}\label{t4}TX=\langle X,\xi \rangle T\xi+\theta(X)\theta^\sharp-\langle JX,\xi \rangle\tr(T)J\xi,\qquad\forall X\in \T M.
		\end{equation}
Taking the scalar product with $\theta^\sharp$ in \eqref{t4} yields for every $X\in \T M$:
\begin{equation}\label{xt}\begin{split}\frac12X(\|\theta\|^2)&=\langle TX,\theta^\sharp\rangle =\langle X,\xi\rangle \langle T\xi,\theta^\sharp\rangle+\theta(X)\|\theta\|^2-\langle JX,\xi\rangle\tr(T)\theta(J\xi)\\
&=\frac12\langle X,\xi\rangle\xi(\|\theta\|^2)+\theta(X)\|\theta\|^2-b \langle JX,\xi\rangle \tr(T),
		\end{split}\end{equation}
		whence 
		\begin{equation}\label{dt}\begin{split}\frac12\di(\|\theta\|^2)=\frac12\xi(\|\theta\|^2)\xi^\flat+\|\theta\|^2\theta+b\tr(T)J\xi^\flat.
		\end{split}\end{equation}
	
	Let us now denote by $\theta_0:=\theta-bJ\xi$ the component of $\theta$ vanishing on $\xi$ and $J\xi$. Using the previous equations and the fact that $\|\theta\|^2=\|\theta_0\|^2+b^2$, we compute:
	\begin{equation*}\begin{split}\frac12\di(\|\theta_0\|^2)&=\frac12\di(\|\theta\|^2-b^2)\stackrel{\eqref{dt}}{=}\frac12\xi(\|\theta\|^2)\xi^\flat+\|\theta\|^2\theta+b\tr(T)J\xi^\flat-b\di b\\
	&=\frac12\xi(\|\theta_0\|^2)\xi^\flat+b\xi(b)\xi^\flat+\|\theta_0\|^2\theta+b^2\theta+b\tr(T)J\xi^\flat-b\di b\\
	&\stackrel{\eqref{t3}}{=}\frac12\xi(\|\theta_0\|^2)\xi^\flat+b\xi(b)\xi^\flat+\|\theta_0\|^2\theta+b^2\theta+b\tr(T)J\xi^\flat-b((TJ\xi)^\flat-\tr(JT)\xi^\flat)\\
	&\stackrel{\eqref{xib}}{=}\frac12\xi(\|\theta_0\|^2)\xi+\|\theta_0\|^2\theta+b^2\theta+b\tr(T)J\xi^\flat-b(TJ\xi)^\flat\\
	&\stackrel{\eqref{t2}}{=}\frac12\xi(\|\theta_0\|^2)\xi^\flat+\|\theta_0\|^2\theta.
	\end{split}\end{equation*}
	This last equation shows that $\|\theta_0\|^2\theta$ vanishes at each critical point of $\|\theta_0\|^2$. In particular, at a point where $\|\theta_0\|^2$ attains its maximum we have $0=\|\|\theta_0\|^2\theta\|^2=\|\theta_0\|^4(b^2+\|\theta_0\|^2)$, thus showing that $\theta_0$ vanishes identically on $M$. 
	
	Consequently, $\theta^\sharp=bJ\xi$, so \eqref{nxi} becomes 
	\begin{equation}\label{nxi1} \nabla_X\xi=b \langle X,\xi\rangle J\xi,\qquad\forall X\in \T M,
	\end{equation}
	which after composing with $J$ yields
	\begin{equation}\label{nxi11} \nabla_XJ\xi=-b \langle X,\xi\rangle\xi,\qquad\forall X\in \T M.
	\end{equation}

	Consider the universal cover $\pi:\widetilde M\to M$, endowed with the Riemannian metric $\widetilde g:=\pi^*g$, with Levi-Civita connection $\widetilde\nabla$, and denote by $\widetilde\xi$ and $\widetilde{J\xi}$ the lifts of $\xi$ and $J\xi$ to $\widetilde M$. If $\widetilde b:=\pi^*(b)$ denotes the pull-back of $b$, \eqref{nxi1} and \eqref{nxi11} become 
	\begin{equation}\label{nxi2} \widetilde\nabla_X\widetilde \xi=\widetilde b\widetilde g(X,\widetilde \xi)\widetilde{J\xi},\qquad \widetilde\nabla_X\widetilde{J\xi}=-\widetilde b\widetilde g(X,\widetilde \xi)\widetilde{\xi},\qquad\forall X\in \T \widetilde M.
	\end{equation}
	In particular we obtain that the distribution spanned by $\widetilde \xi$ and $\widetilde{J\xi}$ is $\widetilde\nabla$-parallel, $\widetilde{J\xi}$ is closed (thus exact) as 1-form on $\widetilde M$, and its integral curves are geodesics, so $\widetilde{J\xi}$ is complete. Since $(\widetilde M,\widetilde g)$ is complete and simply connected, the de Rham decomposition theorem shows that it can be written as a Riemannian product $(M_0,g_0)\times (K,g_K)$ with $(M_0,g_0)$ a complete simply connected Riemannian surface whose tangent bundle in $\widetilde M$ is spanned by $\widetilde \xi$ and $\widetilde{J\xi}$. 
	
	Let $a\in C^\infty(M_0)$ be the primitive of $g_0(\widetilde{J\xi},\cdot)$ which vanishes at some point of $M_0$. Clearly $a:M_0\to\mathbb{R}$ is a submersion. Denote by $C:=a^{-1}(0)$ its level set (which is a 1-dimensional submanifold of $M_0$), and by $(\varphi_t)_{t\in\mathbb{R}}$ the flow of the complete vector field $\widetilde{J\xi}$. Since $\widetilde{J\xi}(a)=1$, we have $\varphi_t(C)=a^{-1}(t)$ for every $t\in\mathbb{R}$. This shows that $(\varphi_t)_*(X)$ is orthogonal to $\widetilde{J\xi}$ for every $t\in\mathbb{R}$ and for every $X\in \T C$. Clearly the map $F:C\times \mathbb{R}\to M_0$ defined by 
$F(x,t):=\varphi_t(x)$ is a smooth bijective immersion, thus a diffeomorphism. In particular $C$ is connected and simply connected, so diffeomorphic to $\mathbb{R}$, whence $M_0$ is diffeomorphic to $\mathbb{R}^2$. We choose a diffeomorphism $f:\mathbb{R}\to C$ and denote by $G:=F\circ(f\times\mathrm{id})$ the corresponding diffeomorphism from $\mathbb{R}^2$ to $M_0$.

Let $(s,t)$ be the standard coordinates in $\mathbb{R}^2$. By construction, the differential of $G$ at every point sends the vector $\frac{\p}{\p s}$ to a multiple of $\widetilde \xi$ and the vector $\frac{\p}{\p t}$ onto $\widetilde{J\xi}$. Consequently, the pull-back metric $G^*g_0$ takes the form $e^{-2\varphi}\di s^2+\di t^2$ for some function $\varphi\in C^\infty(\mathbb{R}^2)$, so $\widetilde\xi$ and $\widetilde{J\xi}$ can  be identified with $e^\varphi\frac{\p}{\p s}$ and $\frac{\p}{\p t}$ respectively. Then $[\widetilde\xi,\widetilde{J\xi}]=-\frac{\p \varphi}{\p t}\widetilde\xi$. On the other hand $\eqref{nxi2}$ gives $[\widetilde\xi,\widetilde{J\xi}]=-\widetilde b\widetilde\xi$, so the pull-back $\widetilde\theta$ to $\widetilde M$ of the Lee form of the Weyl connection $D$ is given by $\widetilde\theta=\widetilde b\widetilde{J\xi}^\sharp=\frac{\p \varphi}{\p t}\di t$. 

Let now $\gamma$ be an element of $\pi_1(M)$. Clearly $\gamma_*$ preserves $\widetilde \xi=e^\varphi\frac{\p}{\p s}$, $\widetilde J\xi=\frac{\p}{\p t}$ and the metric $\widetilde g=e^{-2\varphi}\di s^2+\di t^2+g_K$, so $\gamma$ has the form $\gamma(s,t,x)=(\gamma_0(s),\gamma_1(t),\gamma_K(x))$, where $\gamma_0,\gamma_1\in C^\infty(\mathbb{R})$ and $\gamma_K$ is an isometry of $(K,g_K)$. Moreover, the fact that $(\gamma_0)_*(\frac{\p}{\p t})=\frac{\p}{\p t}$ shows that $\gamma_1'(t)=1$ for all $t\in \mathbb{R}$, so $\gamma_1(t)=t+t_0$ for some $t_0\in \mathbb{R}$.

$(ii) \Longrightarrow (iii)$ Let $\widetilde\nabla$ denote the Levi-Civita connection of the metric $e^{-2\varphi}\di s^2+\di t^2+g_K$ on $\mathbb{R}^2\times K$. Clearly $\T K$ is a $\widetilde\nabla$-parallel complex subbundle of $\T (\mathbb{R}^2\times K)$ of codimension $1$ contained in the kernel of the closed $1$-form of unit length $\di t$. Moreover, the hypotheses imply that $\pi_1(M)$ preserves $\di t$, $\widetilde\nabla$ and $\T K$, so $\di t$ projects to a closed $1$-form of unit length $\eta\in \Omega^1(M)$ and $\T K$ projects to a $\nabla$-parallel complex subbundle of $\T M$ of codimension $1$ contained in the kernel of $\eta$.

$(iii) \Longrightarrow (i)$ Let $\eta\in  \Omega^1(M)$ be a closed 1-form of unit length vanishing on a complex $\nabla$-parallel subbundle $E\subset \T M$ of codimension $1$. Denoting by $\xi:=J\eta^\sharp$ we clearly have $\|\xi\|=1$ and $E^\perp=\mathrm{span}(\xi,J\xi)$. Since $E^\perp$ is $\nabla$-parallel, there exists a 1-form $\alpha\in  \Omega^1(M)$  such that 
\begin{equation}\label{nxi3}\nabla_X\xi=\alpha(X)J\xi,\qquad\nabla_XJ\xi=-\alpha(X)\xi,\qquad\forall X\in \T M.\end{equation}
Moreover, since $\eta=J\xi^\flat$ is closed, we obtain from \eqref{nxi3}  
\begin{equation}\label{nxi4}0=\di\eta=-\alpha\wedge \xi^\flat\end{equation} 
and in particular, applying this last equality to $J\xi$:
\begin{equation}\label{nxi5}\alpha(J\xi)=0.\end{equation}

We claim that the Weyl connection on $(M,g)$ with Lee form $\theta:=J\alpha$ preserves the orthogonal decomposition $\T M=\xi^\perp\oplus\mathbb{R}\xi$. Indeed, for every $X\in \T M$ and $Y\in \xi^\perp$ we can write 
\begin{equation*}\begin{split} \langle D_X\xi,Y\rangle &\stackrel{\eqref{ko}}{=}\langle \nabla_X\xi,Y\rangle+\theta(X)\langle \xi,Y\rangle+\theta(\xi)\langle X,Y\rangle-\langle X,\xi\rangle\theta(Y)\\
&=\alpha(X)\langle J\xi, Y\rangle-\alpha(J\xi)\langle X,Y\rangle+\alpha(JY)\langle X,\xi\rangle\\
&\stackrel{\eqref{nxi5}}{=}(\alpha\wedge\xi^\flat)(JY,X)\stackrel{\eqref{nxi4}}{=}0.
\end{split}\end{equation*}
This shows that $(M,g)$ carries an orientable conformal product structure of rank 1, thus finishing the proof.
\end{proof}

Using Theorem \ref{thm3}, we can easily construct orientable conformal product structures with non-identically zero Lee form on the Riemannian product of certain two-tori with any compact Kähler manifold:

\begin{example}\label{exa}
Let $\varphi:\mathbb{R}^2\to \mathbb{R}$ be a doubly periodic function (\ie satisfying $\varphi(s+1,t)=\varphi(s,t+1)=\varphi(s,t)$ for all $(s,t)\in\mathbb{R}^2$) and let $(K,g_K)$ be any simply connected manifold admitting a discrete co-compact group $\Gamma_K$ of holomorphic isometries (\ie $(K,g_K)$ is the universal cover of a compact Kähler manifold). Then the metric $\widetilde g:=e^{-2\varphi}\di s^2+\di t^2+g_K$ on $\widetilde M:=\mathbb{R}^2\times K$ projects to a Kähler metric $g$ on $M:=\widetilde M/(\mathbb{Z}^2\times\Gamma_K)=T^2\times (K/\Gamma_K)$. By the previous theorem, $(M,g)$ has an orientable conformal product structure of rank 1, whose Lee form with respect to $g$ is the projection to $M$ of $\frac{\p \varphi}{\p t}\di t$. 
\end{example}

On the other hand, not every product of a two-torus and a Kähler manifold carries conformal product structures with non-identically zero Lee form. Indeed, let us first notice the following consequence of our results above.

\begin{cor}\label{cor4.6} If $(M,g,J)$ is a compact Kähler manifold of real dimension $n\ge 4$ carrying an orientable conformal product structure with non-identically zero Lee form, then there exists a geodesic vector field $\zeta$ (\ie satisfying $\nabla_\zeta\zeta=0$) of unit length, such that the distribution spanned by $\zeta$ and $J\zeta$ is $\nabla$-parallel.
\end{cor}

\begin{proof} By Theorem \ref{thm1}, the conformal product structure has rank 1, and by the proof of Theorem \ref{thm3}, there exists a unit vector field $\xi$ on $M$ satisfying \eqref{nxi1} and \eqref{nxi11}. Then the unit vector field $\zeta:=J\xi$ is geodesic and the distribution spanned by $\zeta$ and $J\zeta$ is $\nabla$-parallel.
\end{proof}

By adding a handle to a round 2-sphere and keeping the round metric outside a small disk, V. Matveev \cite{m2024} constructed a metric $g_0$ on the two-torus which carries no geodesic vector field of unit length. Using this result, we then obtain the following:

\begin{prop} \label{vm} Let $(K,g_K)$ be any irreducible compact Kähler manifold of real dimension $m\ge 4$. Then the Riemannian product $(T^2,g_0)\times (K,g_K)$ carries no orientable conformal product structure with non-identically zero Lee form.
\end{prop}

\begin{proof} Assume for a contradiction that $(M,g):=(T^2,g_0)\times (K,g_K)$ carries an orientable conformal product structure $D$ with non-identically zero Lee form.
By Corollary \ref{cor4.6}, there exists a geodesic vector field $\zeta$ of unit length on $(M,g)$ such that the distribution spanned by $\zeta$ and $J\zeta$ is $\nabla$-parallel. The hypothesis implies that the only $\nabla$-parallel distribution of dimension $2$ on $(M,g)$ is the tangent distribution to $T^2$. Restricting $\zeta$ to any leaf $T^2\times\{x\}$ for some $x\in K$, defines a geodesic vector field of unit length on $(T^2,g_0)$, which contradicts the above property of the metric $g_0$.
\end{proof}

\begin{remark} By \cite[Theorem 4.3]{bfm2023}, there exists no LCP structure on compact Kähler manifolds. Thus, a Weyl connection defining a conformal product structure on a compact Kähler manifold is either exact or non-closed. The former case corresponds in Example \ref{exa} to the case where the function $\varphi$ only depends on the variable $t$. In this case the manifold $(M,e^{2\varphi}g)$ from Example \ref{exa} is globally conformally Kähler and has a non-trivial parallel vector field (the projection of $\frac{\p}{\p s}$). It was shown in \cite[Theorem 3.5]{m2015} that all compact conformally Kähler manifolds carrying a non-trivial parallel vector field are obtained by this construction.
\end{remark}

\end{document}